\documentclass[11pt,english]{article}
\usepackage[T1]{fontenc}
\usepackage[latin9]{inputenc}
\usepackage{geometry}
\usepackage{amsmath}
\usepackage{amssymb}
\usepackage{stmaryrd}

\makeatletter
\usepackage{latexsym}
\usepackage{amsfonts}
\usepackage{amsthm}
\usepackage{mathrsfs}
\@ifundefined{definecolor}
 {\usepackage{color}}{}
\@ifundefined{definecolor}{\usepackage{color}}{}

\newcommand{\1}{\mathbf{1}}
\newcommand{\A}{\mathbf{A}}
\newcommand{\B}{\mathbf{B}}
\newcommand{\boldalpha}{\boldsymbol{\alpha}}
\newcommand{\boldlambda}{\boldsymbol{\lambda}}
\newcommand{\boldsigma}{\boldsymbol{\sigma}}
\newcommand{\C}{\mathbb{C}}
\DeclareMathOperator{\diag}{diag}
\newcommand{\G}{\mathcal{G}}
\newcommand{\M}{\mathbf{M}}
\newcommand{\PP}{\mathbf{P}}

\newcommand{\I}{\mathbf{I}}

\newcommand{\U}{\mathbf{U}}
\newcommand{\V}{\mathbf{V}}
\newcommand{\f}{\mathbf{f}}
\newcommand{\g}{\mathbf{g}}
\newcommand{\FF}{\mathbf{F}}
\newcommand{\GG}{\mathbf{G}}

\DeclareMathOperator{\sgn}{sgn}
\newcommand{\sigmamin}{\sigma_{\text{min}}}
\allowdisplaybreaks

\newtheorem{thm}{Theorem}[section]
\newtheorem{cor}[thm]{Corollary}\newtheorem{lem}[thm]{Lemma}

\theoremstyle{definition}
\newtheorem{defn}[thm]{Definition}

\theoremstyle{remark}

\usepackage{babel}
\usepackage{nicefrac}
\usepackage{babel}

\makeatother

\usepackage{babel}
\begin{document}

\title{Spectral Analysis of Non-Hermitian Matrices and Directed Graphs}

\author{Edinah K. Gnang \thanks{Department of Applied Mathematics and Statistics, Johns Hopkins University,
email: egnang1@jhu.edu}, James M. Murphy \thanks{Department of Mathematics, Tufts University, email: jm.murphy@tufts.edu}}
\maketitle
\begin{abstract}
We generalize classical results in spectral graph theory and linear
algebra more broadly, from the case where the underlying matrix is
Hermitian to the case where it is non-Hermitian. New admissibility
conditions are introduced to replace the Hermiticity condition. We
prove new variational estimates of the Rayleigh quotient for non-Hermitian
matrices. As an application, a new Delsarte-Hoffman-type bound on
the size of the largest independent set in a directed graph is developed.
Our techniques consist in quantifying the impact of breaking the Hermitian
symmetry of a matrix and are broadly applicable. 
\end{abstract}

\section{Introduction}

The eigendecomposition is among the most powerful tools for analyzing
Hermitian matrices $\B\in\mathbb{C}^{n\times n}$, i.e. matrices satisfying
$\B^{*}=\B$, where $\B^{*}\left[j,\ell\right]=\overline{\B\left[\ell,j\right]}$.
Several classical results in linear algebra can be derived from decomposing
a Hermitian matrix $\B\in\mathbb{C}^{n\times n}$ as $\B=\U\diag(\boldlambda(\B))\U^{*}$,
where $\U$ is unitary and $\boldlambda(\B)=\left(\lambda_{\ell}(\B)\right)_{0\le\ell<n}\subset\mathbb{R}$.
In particular, under the assumption that $\B$ is Hermitian, variational
estimates on the \emph{Rayleigh quotient} \cite{Horn1990_Matrix}
can be stated in terms of $\boldlambda(\B)$: 
\begin{align}
\forall\f\in\mathbb{C}^{n\times1},\f\neq0,\ \min\left\{ \boldlambda(\B)\right\} \le\frac{\f^{*}\B\f}{\|\f\|_{2}^{2}}\le\max\left\{ \boldlambda(\B)\right\} .\label{eqn:ClassicalRayleigh}
\end{align}

In case that $\B$ is non-Hermitian, the eigenvalues of $\B$ may
be complex, or even worse the eigendecomposition may not exist at
all.

When $\B$ is non-Hermitian but \emph{diagonalizable}, its eigendecomposition
is of the form $\B=\PP\diag(\boldlambda(\B))\PP^{-1}$ for some invertible
matrix $\PP$ and scalars $\boldlambda(\B)$. Compared to the Hermitian
case, the columns of $\PP$ need not form an orthonormal basis. A
different decomposition that is available to all matrices is the \emph{singular
value decomposition (SVD)}: $\B=\U\diag(\boldsigma(\B))\V^{*}$, where
the singular values $\boldsigma(\B)=\left(\sigma_{\ell}(\B)\right)_{0\le\ell<n}$
are non-negative and $\U,\V$ are unitary matrices. However, it need
not be the case that $\U\V^{*}=\U^{*}\V=\I$, which is the primary
contrast with the eigendecomposition.

Powerful tools of linear algebra can be applied to the study of graphs
via spectral graph theory \cite{Chung1997_Spectral}. Indeed, let
$\G=(V,\B)$ be a graph, where $V$ is the set of vertices and $\B\in\{0,1\}^{|V|\times|V|}$
an adjacency matrix such that $\B\left[j,\ell\right]=1$ if there
is an edge between the $j^{th}$ and $\ell^{th}$ nodes. By analyzing
the spectral properties of $\B$, a variety of mathematical ideas
may be adapted to $\G$, including notions of geometry \cite{Mohar1989_Isoperimetric,Mohar1991_Laplacian},
Fourier and wavelet analysis \cite{Coifman2006_DiffusionWavelets,Hammond2011_Wavelets,Shuman2016_Vertex},
random diffusion processes \cite{Coifman2005_Geometric,Coifman2006_DiffusionMaps},
and clusters \cite{Shi2000_Normalized,Ng_2002Spectral}. While these
tools have contributed to a renaissance in the analysis of data, spectral
graph methods almost uniformly require the underlying graph $\G$
to be \emph{undirected}, or in linear algebra terms, $\B$ must be
Hermitian. This is a severe limitation in practice, as a variety of
real data does not lend itself to representation as an undirected
graph, for example social networks \cite{Kwak2010_Twitter}, models
for the spread of contagious disease in a heterogenous population
\cite{Kephart1992_Directed}, and predator-prey relationships \cite{Yorke1973_Predator}.


\subsection{Summary of Contributions}

This article develops new approaches for the analysis of non-Hermitian
matrices. The primary contributions are twofold. First, we prove \emph{a
generalized version of the classical variational estimates on the
Rayleigh quotient.} New \emph{admissibility conditions} are introduced
to replace the Hermiticity condition. Second, the \emph{Delsarte-Hoffman
bound on the size of independent sets in undirected graphs is generalized
to the directed setting}. Our major tool consists in quantifying the
discrepancy between Hermitian and non-Hermitian matrices, and proposing
additional hypotheses in the theorems to address this discrepancy.
The gap between Hermiticity and non-Hermiticity is made precise, and
moreover in the case of the Delsarte-Hoffman bound, the gaps between
$\B$ being Hermitian, diagonalizable (with potentially complex eigenvalues)
and $\B$ arbitrary is considered by analyzing the SVD. Our proof
methods are flexible, and may be applicable to settings not considered
in the present article.

\subsection{Related Work}

Spectral graph theory has been attempted for operators defined on
directed graphs in a variety of contexts, including for the graph
Laplacian \cite{Agaev2005_Spectra,Chung2005_Laplacians,Butler2007_Interlacing,Bauer2012_Normalized,Butler2013_Sharp,Zheng2016_Spectral,Fanuel2018_Magnetic,Chui2018_Representation}
and for nonreversible Markov chains \cite{Fill1991_Eigenvalue}. Combinatorial
results for directed graphs have also been studied \cite{Brualdi2010_Spectra,Kharaghani2015_Hoffman}.
These results do not, however, develop precise characterizations of
the ways in which classical results can be modified in the non-Hermitian
setting. In particular, the admissibility conditions proposed in this
article explicitly illustrate what is lost when a matrix is perturbed
to deviate from Hermiticity, and suggest how to compensate for the
loss of Hermiticity. Moreover, the proposed generalized Delsarte-Hoffman
bound makes no assumptions of normality of the adjacency matrix of
the underlying graph.

\subsection{Notation}

Throughout, bold typography is used to denote matrices and vectors.
Let $\I$ denote the identity matrix with size clear from context.
For a collection of points $\{\alpha_{\ell}\}_{0\le\ell<n}\subset\mathbb{C}$,
let $\diag\left\{ \left(\alpha_{\ell}\right)_{0\le\ell<n}\right\} $
denote the $n\times n$ diagonal matrix with $\ell^{th}$ diagonal
entry $\alpha_{\ell}$. Let $\1_{m\times n}$ and $\mathbf{0}_{m\times n}$
respectively denote the $m\times n$ matrix of all 1s and all 0s.
Let $\B[j,:]$ and $\B[:,\ell]$ denote the $j^{th}$ row and $\ell^{th}$
column of the matrix $\B$, respectively. For matrices $\A,\B$, we
denote the by $\A\circ\B$ the Hadamard product $\left(\A\circ\B\right)\left[j,\ell\right]=\A\left[j,\ell\right]\B\left[j,\ell\right]$.

\section{Generalized Rayleigh Quotient Estimation}

\subsection{Rayleigh Quotient for Complex Diagonalizable Matrices}

Let $\B\in\mathbb{C}^{n\times n}$. There corresponds to $\B$ a (in
general non-symmetric) bilinear form $\left\langle \f,\g\right\rangle _{\B}\mapsto\f^{*}\B\g,$
defined for $\f,\g\in\mathbb{C}^{n\times1}$. The behavior of this
bilinear form can be analyzed in a scale-invariant manner through
the Rayleigh quotient. When $\B$ is Hermitian, (\ref{eqn:ClassicalRayleigh})
states that the Rayleigh quotient $\langle\f,\f\rangle_{\B}/\|\f\|_{2}^{2}$
is controlled for $\f\neq\mathbf{0}_{n\times1}$ by the largest and
smallest eigenvalues of $\B$. We extend this result to the case when
$\B$ has complex eigenvalues.

\begin{thm}\label{thm:Complex_Rayleigh_quotient}

Let $\B\in\mathbb{C}^{n\times n}$ be decomposed as $\B=\U\diag\left\{ \boldlambda\left(\B\right)\right\} \V^{*}$
where $\U\V^{*}=\I$ and $\boldlambda\left(\B\right)\subset\mathbb{C}$.
Write $\lambda_{\ell}(\B)=\left|\lambda_{\ell}(\B)\right|e^{i\theta_{\ell}}$
for $\theta_{\ell}\in[0,2\pi)$. Let $\f,\g\in\mathbb{C}^{n\times1}$
be such that there exist $\FF,\GG\in\mathbb{C}^{n\times1}$ satisfying
$\FF\circ\GG\ne\mathbf{0}_{n\times1}$ and 
\begin{align}
\f=\V\overline{\FF},\ \g=\U\GG\mbox{ s.t. }\forall0\le\ell<n,\;\left(\FF\left[\ell\right]\GG\left[\ell\right]e^{i\theta_{\ell}}\right)\in\mathbb{R}.\label{eqn:AdmissibilityComplexRayleigh}
\end{align}
Then 
\[
\min_{0\le\ell<n}\left\{ \left|\lambda_{\ell}(\B)\right|\sgn(\FF\left[\ell\right]\GG\left[\ell\right]e^{i\theta_{\ell}})\right\} \le\frac{\f^{*}\B\g}{\underset{0\le\ell<n}{\sum}\left|\FF\left[\ell\right]\GG\left[\ell\right]\right|}\le\max_{0\le\ell<n}\left\{ \left|\lambda_{\ell}(\B)\right|\sgn(\FF\left[\ell\right]\GG\left[\ell\right]e^{i\theta_{\ell}})\right\} .
\]

\end{thm}

\begin{proof}Using the eigendecomposition of $\B$ we have 
\begin{align*}
\f^{*}\B\g= & \f^{*}\U\diag\left\{ \boldlambda\left(\B\right)\right\} \V^{*}\g\\
= & \left(\V\bar{\FF}\right)^{*}\U\diag\left\{ \boldlambda\left(\B\right)\right\} \V^{*}\left(\U\GG\right)\\
= & \FF^{\top}\diag\left\{ \boldlambda\left(\B\right)\right\} \GG\\
= & \sum_{0\le\ell<n}\FF\left[\ell\right]\lambda_{\ell}\left(\B\right)\GG\left[\ell\right].
\end{align*}
Writing each eigenvalue in polar form as $\lambda_{\ell}(\B)=\left|\lambda_{\ell}(\B)\right|e^{i\theta_{\ell}}$
and applying the admissibility condition (\ref{eqn:AdmissibilityComplexRayleigh})
yields 
\begin{align*}
 & \min_{0\le\ell<n}\left\{ \left|\lambda_{\ell}(\B)\right|\sgn(\FF\left[\ell\right]\GG\left[\ell\right]e^{i\theta_{\ell}})\right\} \sum_{0\le\ell<n}\left|\FF\left[\ell\right]\GG\left[\ell\right]\right|\\
\le & \sum_{0\le\ell<n}\FF\left[\ell\right]\lambda_{\ell}\left(\B\right)\GG\left[\ell\right]\\
\le & \max_{0\le\ell<n}\left\{ \left|\lambda_{\ell}(\B)\right|\sgn(\FF\left[\ell\right]\GG\left[\ell\right]e^{i\theta_{\ell}})\right\} \sum_{0\le\ell<n}\left|\FF\left[\ell\right]\GG\left[\ell\right]\right|.
\end{align*}

The result follows by algebraic manipulation. \end{proof}

The condition (\ref{eqn:AdmissibilityComplexRayleigh}) is an \emph{admissibility
condition} on the vectors $\f,\g$. In the case where $\B$ is Hermitian,
the eigenvalues of $\B$ are real and $\V=\U$ is unitary. If $\FF\left[\ell\right]\GG\left[\ell\right]\ge0$
for all $\ell$, then 
\[
\sum\limits _{0\le\ell<n}\left|\FF\left[\ell\right]\GG\left[\ell\right]\right|=\sum\limits _{0\le\ell<n}\FF\left[\ell\right]\GG\left[\ell\right]=\langle\f,\g\rangle
\]
since $\U=\V$ is unitary. If $\B$ is Hermitian and moreover $\f=\g$,
which implies that $\overline{\FF}=\GG$, we have

\begin{align*}
 & \sum\limits _{0\le\ell<n}\left|\FF\left[\ell\right]\FF\left[\ell\right]\right|=\|\f\|_{2}^{2},\\
 & \min_{0\le\ell<n}\left|\lambda_{\ell}(\B)\right|\sgn(\FF\left[\ell\right]\FF\left[\ell\right]e^{i\theta_{\ell}})=\min\{\boldlambda(\B)\},\\
 & \max_{0\le\ell<n}\left|\lambda_{\ell}(\B)\right|\sgn(\FF\left[\ell\right]\FF\left[\ell\right]e^{i\theta_{\ell}})=\max\{\boldlambda(\B)\},
\end{align*}

which recovers $(\ref{eqn:ClassicalRayleigh})$. In the general case,
$\f,\g$ must interact in a particular way for Theorem \ref{thm:Complex_Rayleigh_quotient}
to hold, as quantified by the admissibility condition (\ref{eqn:AdmissibilityComplexRayleigh}).
A slightly more general result holds, using the singular value decomopsition modulo
signings of the singular values:

\begin{thm}\label{thm:Complex_Rayleigh_quotient-1}

Let $\B\in\mathbb{C}^{n\times n}$ be decomposed as $\B=\U\diag\left\{ \boldsigma\left(\B\right)\right\} \V^{*}$
where $\U\U^{*}=\I=\V\V^{*}$ and $\boldsigma\left(\B\right)\subset\mathbb{C}$.
Write $\sigma_{\ell}(\B)=\left|\sigma_{\ell}(\B)\right|e^{i\theta_{\ell}}$
for $\theta_{\ell}\in\left\{ 0,\pi\right\} $. Let $\f,\g\in\mathbb{C}^{n\times1}$
be such that there exist $\FF,\GG\in\mathbb{C}^{n\times1}$ satisfying
$\FF\circ\GG\ne\mathbf{0}_{n\times1}$ and 
\begin{align}
\f=\U\overline{\FF},\ \g=\V\GG\mbox{ s.t. }\forall0\le\ell<n,\;\left(\FF\left[\ell\right]\GG\left[\ell\right]e^{i\theta_{\ell}}\right)\in\mathbb{R}.\label{eqn:AdmissibilityComplexRayleigh-1}
\end{align}
Then 
\[
\min_{0\le\ell<n}\left\{ \left|\sigma_{\ell}(\B)\right|\sgn(\FF\left[\ell\right]\GG\left[\ell\right]e^{i\theta_{\ell}})\right\} \le\frac{\f^{*}\B\g}{\underset{0\le\ell<n}{\sum}\left|\FF\left[\ell\right]\GG\left[\ell\right]\right|}\le\max_{0\le\ell<n}\left\{ \left|\sigma_{\ell}(\B)\right|\sgn(\FF\left[\ell\right]\GG\left[\ell\right]e^{i\theta_{\ell}})\right\} .
\]

\end{thm}

\begin{proof}Using the Singular value decomopsition modulo signings
of the singular values of $\B$ we have 
\begin{align*}
\f^{*}\B\g= & \f^{*}\U\diag\left\{ \boldsigma\left(\B\right)\right\} \V^{*}\g\\
= & \left(\U\bar{\FF}\right)^{*}\U\diag\left\{ \boldsigma\left(\B\right)\right\} \V^{*}\left(\V\GG\right)\\
= & \FF^{\top}\diag\left\{ \boldsigma\left(\B\right)\right\} \GG\\
= & \sum_{0\le\ell<n}\FF\left[\ell\right]\sigma_{\ell}\left(\B\right)\GG\left[\ell\right].
\end{align*}
Writing each singular value in polar form as $\sigma_{\ell}(\B)=\left|\sigma_{\ell}(\B)\right|e^{i\theta_{\ell}}$
and applying the admissibility condition (\ref{eqn:AdmissibilityComplexRayleigh-1})
yields 
\begin{align*}
 & \min_{0\le\ell<n}\left\{ \left|\sigma_{\ell}(\B)\right|\sgn(\FF\left[\ell\right]\GG\left[\ell\right]e^{i\theta_{\ell}})\right\} \sum_{0\le\ell<n}\left|\FF\left[\ell\right]\GG\left[\ell\right]\right|\\
\le & \sum_{0\le\ell<n}\FF\left[\ell\right]\sigma_{\ell}\left(\B\right)\GG\left[\ell\right]\\
\le & \max_{0\le\ell<n}\left\{ \left|\sigma_{\ell}(\B)\right|\sgn(\FF\left[\ell\right]\GG\left[\ell\right]e^{i\theta_{\ell}})\right\} \sum_{0\le\ell<n}\left|\FF\left[\ell\right]\GG\left[\ell\right]\right|.
\end{align*}

The result follows by algebraic manipulation. \end{proof}

\subsection{Illustration of Admissible Vectors Via Index Rotations}

Admissible vectors may be constructed as follows. Consider the matrix
transformation prescribed by performing a rotation to the entry indices
of a matrix $\mathbf{A}\in\mathbb{C}^{n\times n}$: 
\[
\left(\mathbf{A}^{\text{R}_{\theta}}\right)\left[j,\ell\right]=
\]
\[
\begin{cases}
\begin{array}{c}
\mathbf{A}\left[\left(j-\left\lfloor \frac{n}{2}\right\rfloor \right)\cos\theta+\left(\left\lfloor \frac{n}{2}\right\rfloor -\ell\right)\sin\theta+\left\lfloor \frac{n}{2}\right\rfloor ,\left(j-\left\lfloor \frac{n}{2}\right\rfloor \right)\sin\theta-\left(\left\lfloor \frac{n}{2}\right\rfloor -\ell\right)\cos\theta+\left\lfloor \frac{n}{2}\right\rfloor \right],\:n\text{ odd}\\
\\
\mathbf{A}\left[\left(j-\frac{n-1}{2}\right)\cos\theta+\left(\frac{n-1}{2}-\ell\right)\sin\theta+\frac{n-1}{2},\left(j-\frac{n-1}{2}\right)\sin\theta-\left(\frac{n-1}{2}-\ell\right)\cos\theta+\frac{n-1}{2}\right],\:n\text{ even}
\end{array}\end{cases}
\]
where rotation angles are restricted to $\theta\in\left\{ 0,\frac{\pi}{2},\pi,\frac{3\pi}{2}\right\} $.
Together with the transpose, these transformations generate the dihedral
group of order $8$. Given a matrix $\A$, we have \[
\mathbf{A}^{\text{R}_{0}}=\mathbf{A}=\left(\begin{array}{rrr}
a_{00} & a_{01} & a_{02}\\
a_{10} & a_{11} & a_{12}\\
a_{20} & a_{21} & a_{22}
\end{array}\right),\]
\[
\mathbf{A}^{\text{R}_{\frac{\pi}{2}}}=\left(\begin{array}{rrr}
a_{20} & a_{10} & a_{00}\\
a_{21} & a_{11} & a_{01}\\
a_{22} & a_{12} & a_{02}
\end{array}\right),\mathbf{A}^{\text{R}_{\pi}}=\left(\begin{array}{rrr}
a_{22} & a_{21} & a_{20}\\
a_{12} & a_{11} & a_{10}\\
a_{02} & a_{01} & a_{00}
\end{array}\right),\mathbf{A}^{\text{R}_{\frac{3\pi}{2}}}=\left(\begin{array}{rrr}
a_{02} & a_{12} & a_{22}\\
a_{01} & a_{11} & a_{21}\\
a_{00} & a_{10} & a_{20}
\end{array}\right).
\]
Furthermore, these transformations preserve unitarity:

\begin{lem}Suppose $\mathbf{A}$ is unitary. Then $\forall\:\theta\in\left\{ 0,\frac{\pi}{2},\pi,\frac{3\pi}{2}\right\} $,
$\A^{\text{R}_{\theta}}$ is also unitary. \end{lem}

\begin{proof}

It suffices to prove that $\left(\A^{\text{R}_{\frac{\pi}{2}}}\right)^{*}\left(\A^{\text{R}_{\frac{\pi}{2}}}\right)=\mathbf{I}$.
Recall that $\A$ unitary means that 
\[
\A^{*}\A=\I=\A\A^{*},
\]
and so $\A$ has rows and columns with $\ell^{2}$ norm equal to 1.
Since the entries of the $i^{th}$ row of $\A^{\text{R}_{\frac{\pi}{2}}}$
correspond to a permutation of the entries of the $(n-1-j)^{th}$
column of $\mathbf{A}$ it follows that $\A^{\text{R}_{\frac{\pi}{2}}}$
also has rows and columns with $\ell^{2}$ norm equal to 1. Furthermore
since every column of $\mathbf{A}$ undergoes the same permutation
when converted into a row of $\A^{\text{R}_{\frac{\pi}{2}}}$ it follows
that $\forall\,0\le j<\ell<n,$ 
\[
\quad\A^{\text{R}_{\frac{\pi}{2}}}\left[j,:\right]\left(\A^{\text{R}_{\frac{\pi}{2}}}\left[\ell,:\right]\right)^{*}=\mathbf{A}\left[n-1-j,:\right]\left(\mathbf{A}\left[n-1-\ell,:\right]\right)^{*}=\begin{cases}
\begin{array}{cc}
1, & j=\ell,\\
0, & \text{otherwise},
\end{array}\end{cases}
\]
thus completing the proof. The other angles follow similarly.

\end{proof}

Moreover, index rotations obey a convenient multiplicative identity:

\begin{lem}\label{lem:RotationMultiplication}Let $\A,\B\in\mathbb{C}^{n\times n}$.
Then $\left(\mathbf{A}\mathbf{B}\right)^{\text{R}_{\frac{\pi}{2}}}=\mathbf{B}^{\top}\left(\left(\mathbf{A}^{\top}\right)^{\text{R}_{\frac{3\pi}{2}}}\right)^{\top}$.
\end{lem}

\begin{proof}We prove the case when $n$ is odd; the argument is
analogous for $n$ even. Recall that 
\[
\left(\A\B\right)\left[j,\ell\right]=\sum_{0\le t<n}\A\left[j,t\right]\B\left[t,\ell\right].
\]
Applying the rotation operator,

\begin{align*}
 & \left(\A\B\right)^{\text{R}_{\theta}}\left[j,\ell\right]=\\
 & \sum_{0\le t<n}\A\left[\left(j-\left\lfloor \frac{n}{2}\right\rfloor \right)\cos\theta+\left(\left\lfloor \frac{n}{2}\right\rfloor -\ell\right)\sin\theta+\left\lfloor \frac{n}{2}\right\rfloor ,t\right]\B\left[t,\left(j-\left\lfloor \frac{n}{2}\right\rfloor \right)\sin\theta-\left(\left\lfloor \frac{n}{2}\right\rfloor -\ell\right)\cos\theta+\left\lfloor \frac{n}{2}\right\rfloor \right].
\end{align*}
In particular, if $\theta=\frac{\pi}{2}$ we have 
\[
\left(\A\B\right)^{\text{R}_{\frac{\pi}{2}}}\left[j,\ell\right]=\sum_{0\le t<n}\B\left[t,j\right]\A\left[\left(n-1\right)-\ell,t\right].
\]
Similarly, 
\begin{align*}
 & \left(\mathbf{A}^{\top}\right)^{\text{R}_{\theta}}\left[j,\ell\right]=\\
 & \mathbf{A}\left[\left(\ell-\left\lfloor \frac{n}{2}\right\rfloor \right)\cos\left(-\theta\right)+\left(\left\lfloor \frac{n}{2}\right\rfloor -j\right)\sin\left(-\theta\right)+\left\lfloor \frac{n}{2}\right\rfloor ,\left(\ell-\left\lfloor \frac{n}{2}\right\rfloor \right)\sin\left(-\theta\right)-\left(\left\lfloor \frac{n}{2}\right\rfloor -j\right)\cos\left(-\theta\right)+\left\lfloor \frac{n}{2}\right\rfloor \right]
\end{align*}
In particular, for $\theta=\frac{3\pi}{2}$ we have 
\[
\left(\mathbf{A}^{\top}\right)^{\text{R}_{\frac{3\pi}{2}}}\left[j,\ell\right]=\mathbf{A}\left[\left(n-1\right)-j,\ell\right],
\]
from which the desired claim follows. \end{proof}

We now establish the existence of a class of admissible vectors.

\begin{thm}\label{cor:tumble-2} Let $\M\in\mathbb{R}^{n\times n}$
be Hermitian such that $\M^{\text{R}_{\frac{\pi}{2}}}$ is diagonalizable
with 
\[
\M^{\text{R}_{\frac{\pi}{2}}}=\U\diag\left\{ \boldlambda\left(\M^{\text{R}_{\frac{\pi}{2}}}\right)\right\} \V^{*},\ \U\V^{*}=\I.
\]
Write the eigenvalues of $\M^{\text{R}_{\frac{\pi}{2}}}$ in the polar
form $\lambda_{\ell}\left(\M^{\text{R}_{\frac{\pi}{2}}}\right)=\left|\lambda_{\ell}\left(\M^{\text{R}_{\frac{\pi}{2}}}\right)\right|e^{i\theta_{\ell}},$
$\ell=0,\dots,n-1$. Then for every $0\le j<n$ there are admissible
vector pairs $\f,\g\in\mathbb{C}^{n\times1}$ for which 
\[
\f=\V\bar{\FF},\:\g=\U\GG,\:\FF\circ\GG\ne\mathbf{0}_{n\times1}\mbox{ s.t. }\forall\:0\le\ell<n,\;\left(\FF\left[\ell\right]\GG\left[\ell\right]e^{i\theta_{\ell}}\right)\in\mathbb{R}
\]
and moreover, 
\[
\left(\min_{0\le\ell<n}\lambda_{\ell}\left(\M\right)\right)\|\f\|_{2}^{2}\le\lambda_{j}\left(\M^{\text{R}_{\frac{\pi}{2}}}\right)\FF\left[j\right]\GG\left[j\right]\le\left(\max_{0\le\ell<n}\lambda_{\ell}\left(\M\right)\right)\|\f\|_{2}^{2}.
\]
\end{thm}

\begin{proof} For an arbitrary matrix $\mathbf{B}\in\mathbb{C}^{n\times n}$
the following identity follows from Lemma 2.3 
\[
\mathbf{f}^{*}\B\mathbf{g}=\mathbf{f}^{*}\left(\B^{\top}\right)^{\text{R}_{\frac{\pi}{2}}}\left(\sum_{0\le\ell<n}\I\left[:,\ell\right]\I\left[n-1-\ell,:\right]\right)\mathbf{g},
\]
since for any matrix $\mathbf{B}$ we have 
\begin{align*}
\left(\mathbf{B}^{\top}\right)^{\text{R}_{\frac{\pi}{2}}}\left(\sum_{0\le\ell<n}\mathbf{I}\left[:,\ell\right]\mathbf{I}\left[n-1-\ell,:\right]\right)= & \left(\mathbf{B}^{\top}\right)^{\text{R}_{\frac{\pi}{2}}}\left(\left(\mathbf{I}^{\top}\right)^{\text{R}_{\frac{3\pi}{2}}}\right)^{\top}\\
= & \left(\left(\left(\mathbf{B}^{\top}\right)^{\text{R}_{\frac{\pi}{2}}}\right)^{\top}\right)^{\top}\left(\left(\mathbf{I}^{\top}\right)^{\text{R}_{\frac{3\pi}{2}}}\right)^{\top}\\
= & \left(\mathbf{I}\left(\left(\mathbf{B}^{\top}\right)^{\text{R}_{\frac{\pi}{2}}}\right)^{\top}\right)^{\text{R}_{\frac{\pi}{2}}}\\
= & \left(\left(\left(\mathbf{B}^{\top}\right)^{\text{R}_{\frac{\pi}{2}}}\right)^{\top}\right)^{\text{R}_{\frac{\pi}{2}}}\\
= & \mathbf{B}.
\end{align*}
In particular, for a real Hermitian $\M$, 
\[
\mathbf{f}^{*}\M\mathbf{f}=\mathbf{f}^{*}\M^{\text{R}_{\frac{\pi}{2}}}\left(\sum_{0\le\ell<n}\I\left[:,\ell\right]\I\left[n-1-\ell,:\right]\right)\mathbf{f}.
\]
For some fixed $j$ let 
\[
\f=\left(\sum_{0\le\ell<n}\I\left[:,\ell\right]\I\left[n-1-\ell,:\right]\right)\U\left[:,j\right],\ \mathbf{g}=\U\left[:,j\right].
\]
Note that there exist a unique pair of vectors $\bar{\FF},\GG\in\mathbb{C}^{n\times1}$
such that $\f=\V\bar{\FF}\;\text{ and }\;\g=\U\GG$. In particular,
$\GG=(0,0,\dots,0,1,0,\dots,0)^{\top}$, with the 1 in the $j^{th}$
coordinate. So, to show admissibility of the vector pair $\mathbf{f},\mathbf{g}\in\mathbb{C}^{n\times1}$,
it thus suffices to show $\FF[j]\GG[j]e^{i\theta_{j}}\in\mathbb{R}$.
This follows from the fact that 
\[
\f^{*}\M^{\text{R}_{\frac{\pi}{2}}}\g=\f^{*}\M^{\text{R}_{\frac{\pi}{2}}}\left(\sum_{0\le\ell<n}\I\left[:,\ell\right]\I\left[n-1-\ell,:\right]\right)\mathbf{f}=\left(\f^{*}\M\f\right)\in\mathbb{R}
\]
and also 
\[
\f^{*}\M^{\text{R}_{\frac{\pi}{2}}}\g=\FF^{\top}\diag(\{\boldlambda(\M^{\text{R}_{\frac{\pi}{2}}})\})\GG=\FF[j]\lambda_{j}\left(\M^{\text{R}_{\frac{\pi}{2}}}\right)\GG[j].
\]
Moreover, an application of the variational Rayleigh quotient estimates
for $\M$ yields
\[
\min_{0\le\ell<n}\lambda_{\ell}\left(\M\right)\|\f\|_{2}^{2}\le\FF[j]\lambda_{j}\left(\M^{\text{R}_{\frac{\pi}{2}}}\right)\GG[j]\le\left(\max_{0\le\ell<n}\lambda_{\ell}\left(\M\right)\right)\|\f\|_{2}^{2},
\]
thus concluding the proof.

\end{proof}

We remark that the matrix $\M^{\text{R}_{\frac{\pi}{2}}}$ is a \emph{perHermitian
matrix} \cite{Golub2012_Matrix}.  Note that the rotation operator preserves the eigenvalues of a Hermitian matrix, though the eigenvectors change in general.

\begin{cor}\label{cor:tumble-1} Let $\M$ be a Hermitian matrix
with spectral decomposition $\M=\V^{*}\diag\{\boldlambda(\M)\}\V$,
$\V^{*}\V=\I$. Then $\M^{\text{R}_{\frac{\pi}{2}}}=\mathbf{V}^{\top}\diag\left\{ \boldlambda\left(\M\right)\right\} \left(\left(\mathbf{V}^{*\top}\right)^{\text{R}_{\frac{3\pi}{2}}}\right)^{\top}$.\end{cor}

\begin{proof} Applying Lemma \ref{lem:RotationMultiplication} to
the factorization $\V^{*}\left(\diag\{\boldlambda(\M)\}\V\right)$
yields the desired result. \end{proof}

\section{Estimating Independent Set Cardinalities in Directed Graphs}

As an application of our method to graph theory, we develop estimates
on the size of the largest independent set in certain directed graphs.
We will consider the inner product on $n\times n$ matrices $\langle\B,\M\rangle=\text{tr}(\B^{*}\M)$,
which has associated norm 
\[
\|\B\|=\|\B\|_{\text{Fro}}=\sqrt{\sum_{0\le j,\ell<n}|\B[j,\ell]|^{2}}.
\]

\begin{defn} Let $\G$ be a directed, unweighted graph on $n$ vertices.
The graph $\mathcal{G}$ has \emph{adjacency matrix} $\B\in\{0,1\}^{n\times n}$
where $\B\left[i,j\right]=1$ if and only if there is a directed edge
from the $i^{th}$ node to the $j^{th}$ node in $\mathcal{G}$. The
graph $\mathcal{G}$ is \emph{$d$-regular} if all row and columns
sums are equal to $d$. The graph $\mathcal{G}$ is \emph{undirected}
if it has symmetric adjacency matrix. \end{defn}

The \emph{Delsarte-Hoffman bound} \cite{Delsarte1973_Algebraic,Hoffman2003_Eigenvalues}
is a classical estimate on the cardinality of the largest independent
set of an undirected, unweighted, $d$-regular graph in terms of the
spectrum of its adjacency matrix. We state it and provide a proof
for completeness.

\begin{thm}\label{thm:Delsarte-HoffmanUndirected}(Undirected Delsarte-Hoffman
Bound). Let $\B\in\{0,1\}^{n\times n}$ be the adjacency matrix of
an undirected $d$-regular graph $\G$ with spectral decomposition
\[
\B^{k}=\U\diag(\boldlambda(\B))^{k}\U^{*},\ k\in\left\{ 0,1\right\} 
\]
where $\boldlambda(\B)\subset\mathbb{R}$. Let $I$ be the indices
of an independent set in $\G$. Then

\begin{equation}
\frac{|I|}{n}\le\frac{-\min\left\{ \boldlambda\left(\B\right)\right\} }{d-\min\left\{ \boldlambda\left(\B\right)\right\} }.\label{eqn:HermitianDelsarte-HoffmanBound}
\end{equation}
\end{thm} \begin{proof}

Note that by $d$-regularity, $\B$ has an eigenvalue of $d$; without
loss of generality, let $\lambda_{0}(\B)=d$. Then $\B=\frac{d}{n}\1_{n\times n}+\left(\I-\frac{\1_{n\times n}}{n}\right)\B.$
Thus, for all $\f\in\C^{n\times1}$, 
\[
\left\langle \f\,\f^{*}, \B\right\rangle =d\left\langle \f\,\f^{*}, \frac{\1_{n\times n}}{n}\right\rangle +\left\langle \f\,\f^{*}, \left(\I-\frac{\1_{n\times n}}{n}\right)\B \right\rangle .
\]
Let $\B=\U\diag\{\bold\lambda(\B)\}\U^{*}$ where $\U\U^{*}=\I$.
We analyze the second summand on the right-hand side as follows: 
\begin{align*}
\left\langle \f\,\f^{*},\left(\I-\frac{\1_{n\times n}}{n}\right)\B\right\rangle = & \f^{*}\left(\I-\frac{\1_{n\times n}}{n}\right)\B\f\\
= & \f^{*}\left(\I-\frac{\1_{n\times n}}{n}\right)\U\diag(\boldlambda(\B))\U^{*}\f\\
= & \f^{*}\left(\U\diag(\boldlambda(\B))\U^{*}-\frac{1}{n}\sum_{0\le\ell<n}\lambda_{\ell}(\B)\1_{n\times n}\U[:,\ell]\U^{*}[\ell,:]\right)\f\\
= & \f^{*}\left(\sum_{0\le\ell<n}\lambda_{\ell}(\B)\U[:,\ell]\U^{*}[\ell,:]-\lambda_{0}(\B)\U[:,0]\U^{*}[0,:]\right)\f\\
= & \f^{*}\left(\sum_{1\le\ell<n}\lambda_{\ell}(\B)\U[:,\ell]\U^{*}[\ell,:]\right)\f\\
= & \sum_{1\le\ell<n}\lambda_{\ell}(\B)(\f^{*}\U[:,\ell])^{2}\\
\ge & \min\{\boldlambda(\B)\}\sum_{1\le\ell<n}(\f^{*}\U[:,\ell])^{2}\\
= & \min\{\boldlambda(\B)\}\sum_{1\le\ell<n}\f^{*}\U[:,\ell]\U^{*}[\ell,:]\f\\
= & \min\{\boldlambda(\B)\}\left\langle\f\,\f^{*}, \left(\I-\frac{\1_{n\times n}}{n}\right)\right\rangle .
\end{align*}
Hence, 
\begin{align*}
\left\langle \f\,\f^{*},\B\right\rangle \ge & d\left\langle \f\,\f^{*},\frac{\1_{n\times n}}{n}\right\rangle +\min\left\{ \boldlambda\left(\B\right)\right\} \left\langle \f\,\f^{*},\left(\I-\frac{\1_{n\times n}}{n}\right)\right\rangle \\
= & \frac{d}{n}\|\f\|_{2}^{4}+\min\left\{ \boldlambda\left(\B\right)\right\} \left(\|\f\|_{2}^{2}-\frac{1}{n}\|\f\|_{2}^{4}\right).
\end{align*}
To every independent set $I$, there is a corresponding indicator
vector $\f=\1_{I}$ for which by definition $\langle\B,\f\,\f^{*}\rangle=\f^{*}\B\f=0$.
For such an indicator vector $\f$, it follows that 
\[
0\ge\frac{d}{n}\|\f\|_{2}^{4}-\frac{\min\left\{ \boldlambda\left(\B\right)\right\} }{n}\|\f\|_{2}^{4}+\min\left\{ \boldlambda\left(\B\right)\right\} \|\f\|_{2}^{2}.
\]
Noting that $\left|I\right|=\|\f\|_{2}^{2}$, we get 
\[
0\ge\frac{d}{n}\left|I\right|^{2}-\frac{\min\left\{ \boldlambda\left(\B\right)\right\} }{n}\left|I\right|^{2}+\min\left\{ \boldlambda\left(\B\right)\right\} \left|I\right|.
\]
\[
\implies\frac{\left|I\right|}{n}\le\frac{-\min\left\{ \boldlambda\left(\B\right)\right\} }{d-\min\left\{ \boldlambda\left(\B\right)\right\} }.
\]
thus completing the proof. \end{proof}

The condition on the maximum size on an independent set may be characterized
as the maximum value $\|\1_{I}\|_{2}^{2}$, where $\1_{I}^{*}\B\1_{I}=0$
for some $I\subset V$. The spectral decomposition of $\B$ decouples
the rank-one matrix $\1_{n\times n}$ associated with the eigenvalue
$\lambda_{0}(\B)=d$, from whence the analysis flows. Note that the
Delsarte-Hoffman bound is sharp for the complete bipartite graph having
$n$ vertices in each partition, since in this case 
\[
\max_{I\text{ independent}}|I|=n,\min\{\boldlambda(\B)\}=-n,d=n.
\]


\subsection{Directed Delsarte-Hoffman Bound}

We now consider independent sets in \emph{directed} regular graphs
and broaden the scope to adjacency matrices whose entries are not
necessarily binary.

\begin{defn} Let $\mathcal{G}$ be a directed graph with $n$ nodes.
A matrix $\B\in\mathbb{C}^{n\times n}$ is a \emph{pseudo-adjacency
matrix} for $\mathcal{G}$ if $\B_{ij}=0$ whenever there is no directed
edge from the $i^{th}$ node to the $j^{th}$ node in $\mathcal{G}$.
The pseudo-adjacency matrix $\B$ is said to be \emph{$\delta$-regular}
if all row and column sums of $\B$ are equal to $\delta$. \end{defn}

We develop Delsarte-Hoffman-type bounds based on the spectral decomposition
and the singular value decomposition of the (non-Hermitian) pseudo-adjacency
matrix.

\begin{thm}\label{thm:Delsarte-HoffmanDirected}

Let $\B\in\mathbb{C}^{n\times n}$ denote a $\delta$-regular pseudo-adjacency
matrix for a graph $\G$. Let $\B$ be decomposed as $\B=\U\diag\left\{ \boldlambda\left(\B\right)\right\} \V^{*}$
where $\U\V^{*}=\I$. Let $\lambda_{\ell}(\B)=\alpha_{\ell}e^{i\theta_{\ell}}$,
$\alpha_{\ell}\in\mathbb{R},\theta_{\ell}\in[0,2\pi)$ be a polar
form of the $\ell$-th eigenvalue of $\B$. Let $\boldalpha(\B)=(\alpha_{0},\dots,\alpha_{n-1})$.
Let $\f,\g\in\{0,1\}^{n\times1}$ be such that there exist $\FF,\GG\in\mathbb{C}^{n\times1}$
satisfying 
\begin{align}
\f=\V\overline{\FF},\ \g=\U\GG\mbox{ such that }\forall0\le\ell<n,\:\left(\FF\left[\ell\right]\GG\left[\ell\right]e^{i\theta_{\ell}}\right)\ge0.\label{eqn:AdmissibilityConditionEigDecomp}
\end{align}
Then if $\f$ and $\g$ denote respectively indicator vectors for
rows and columns associated with a rectangular 0 block in $\B$, 
\[
\frac{\|\f\|_{2}^{2}\|\g\|_{2}^{2}}{n}\le\dfrac{-\min\{\boldalpha(\B)\}\sum\limits _{0\le\ell<n}\FF\left[\ell\right]\GG\left[\ell\right]e^{i\theta_{\ell}}}{\delta-\min\{\boldalpha(\B)\}}.
\]
\end{thm}

\begin{proof}

Note that by $\delta$-regularity, $\B$ has an eigenvalue of $\delta$;
without loss of generality, let $\lambda_{0}(\B)=\delta$. Then we
have that $\U[:,0]=\V[:,0]=\frac{1}{\sqrt{n}}(1,1,\dots,1)^{\top}$.
Hence, $\B=\frac{\delta}{n}\1_{n\times n}+\left(\I-\frac{\1_{n\times n}}{n}\right)\B.$
Thus, for all $\f,\g\in\C^{n\times1}$ subject to (\ref{eqn:AdmissibilityConditionEigDecomp}),
\[
\left\langle \f\,\g^{*},\B\right\rangle =\delta\left\langle \f\,\g^{*},\frac{\1_{n\times n}}{n}\right\rangle +\left\langle \f\,\g^{*},\left(\I-\frac{\1_{n\times n}}{n}\right)\B\right\rangle .
\]

We analyze the second term of the right hand side as follows:

\begin{align*}
\left\langle \f\,\g^{*},\left(\I-\frac{\1_{n\times n}}{n}\right)\B\right\rangle = & \f^{*}\left(\I-\frac{\1_{n\times n}}{n}\right)\B\g\\
= & \f^{*}\left(\I-\frac{\1_{n\times n}}{n}\right)\U\diag(\boldlambda(\B))\V^{*}\g\\
= & \f^{*}\left(\U\diag(\boldlambda(\B))\V^{*}-\frac{1}{n}\sum_{0\le\ell<n}\lambda_{\ell}(\B)\1_{n\times n}\U[:,\ell]\V^{*}[\ell,:]\right)\g\\
= & \f^{*}\left(\sum_{0\le\ell<n}\lambda_{\ell}(\B)\U[:,\ell]\V^{*}[\ell,:]-\lambda_{0}(\B)\U[:,0]\V^{*}[0,:]\right)\g\\
= & \f^{*}\left(\sum_{1\le\ell<n}\lambda_{\ell}(\B)\U[:,\ell]\V^{*}[\ell,:]\right)\g\\
= & \FF^{\top}\V^{*}\left(\sum_{1\le\ell<n}\lambda_{\ell}(\B)\U[:,\ell]\V^{*}[\ell,:]\right)\U\GG\\
= & \sum_{1\le\ell<n}\FF\left[\ell\right]\lambda_{\ell}(\B)\GG\left[\ell\right]\\
= & \sum_{1\le\ell<n}\alpha_{\ell}\FF\left[\ell\right]\GG\left[\ell\right]e^{i\theta_{\ell}}\\
\ge & \min\{\boldalpha(\B)\}\sum_{1\le\ell<n}\FF\left[\ell\right]\GG\left[\ell\right]e^{i\theta_{\ell}}\\
= & \min\{\boldalpha(\B)\}\left(-\FF\left[0\right]\GG\left[0\right]+\sum_{0\le\ell<n}\FF\left[\ell\right]\GG\left[\ell\right]e^{i\theta_{\ell}}\right)\\
= & \min\{\boldalpha(\B)\}\left(-\frac{1}{n}\|\f\|_{2}^{2}\|\g\|_{2}^{2}+\sum_{0\le\ell<n}\FF\left[\ell\right]\GG\left[\ell\right]e^{i\theta_{\ell}}\right).\\
\end{align*}
Note that $\FF\left[0\right]=\frac{1}{\sqrt{n}}\|\f\|_{1}=\frac{1}{\sqrt{n}}\|\f\|_{2}^{2}$
follows from the observation that $\f$ takes values in $\{0,1\},\U^{*}\f=\U^{*}\V\overline{\FF}=\overline{\FF}$,
and $\U[:,0]=\frac{1}{\sqrt{n}}(1,\dots,1)^{\top}$. That $\GG\left[0\right]=\frac{1}{\sqrt{n}}\|\g\|_{2}^{2}$
follows similarly. Hence, 
\begin{align*}
\left\langle \f\,\g^{*},\B\right\rangle \ge & \delta\left\langle \f\,\g^{*},\frac{\1_{n\times n}}{n}\right\rangle +\min\{\boldalpha(\B)\}\left(-\frac{1}{n}\|\f\|_{2}^{2}\|\g\|_{2}^{2}+\sum_{0\le\ell<n}\FF\left[\ell\right]\GG\left[\ell\right]e^{i\theta_{\ell}}\right)\\
= & \frac{\delta}{n}\|\f\|_{2}^{2}\|\g\|_{2}^{2}+\min\{\boldalpha(\B)\}\left(-\frac{1}{n}\|\f\|_{2}^{2}\|\g\|_{2}^{2}+\sum_{0\le\ell<n}\FF\left[\ell\right]\GG\left[\ell\right]e^{i\theta_{\ell}}\right).\\
\end{align*}
Since $\f$ and $\g$ are indicator vectors for rows and columns associated
with a rectangular 0 block in $\B$, $\f^{*}\B\g=0$. It follows that
\[
0\ge\frac{\delta}{n}\|\f\|_{2}^{2}\|\g\|_{2}^{2}+\min\{\boldalpha(\B)\}\left(-\frac{1}{n}\|\f\|_{2}^{2}\|\g\|_{2}^{2}+\sum_{0\le\ell<n}\FF\left[\ell\right]\GG\left[\ell\right]e^{i\theta_{\ell}}\right),
\]
from whence the result follows by algebraic manipulation.

\end{proof}

In the case that $\B$ is the adjacency matrix of a $d$-regular graph,
the result may be interpreted as a generalization of Theorem \ref{thm:Delsarte-HoffmanUndirected}.

\begin{cor}(Directed Delsarte-Hoffman Bound)\label{cor:Delsarte-HoffmanDirected}
Let $\B\in\{0,1\}^{n\times n}$ denote a $d$-regular adjacency matrix
for a graph $\G$ on $n$ vertices. Let $\B$ be decomposed as $\B=\U\diag\left\{ \boldlambda\left(\B\right)\right\} \V^{*}$
where $\U\V^{*}=\I$. Let $\lambda_{\ell}(\B)=\alpha_{\ell}e^{i\theta_{\ell}},$
$\alpha_{\ell}\in\mathbb{R},\theta_{\ell}\in[0,2\pi)$ be a polar
decomposition of $\boldlambda(\B)$. Let $\boldalpha(\B)=(\alpha_{1},\dots,\alpha_{n})$.
Let $\f,\g\in\{0,1\}^{n\times1}$ be such that there exist $\FF,\GG\in\mathbb{C}^{n\times1}$
satisfying 
\begin{align}
\f=\V\overline{\FF},\ \g=\U\GG\mbox{ such that }\forall0\le\ell<n,\:\left(\FF\left[\ell\right]\GG\left[\ell\right]e^{i\theta_{\ell}}\right)\ge0.\label{eqn:AdmissibilityConditionEigDecomp}
\end{align}
Then if $\f$ and $\g$ denote respectively indicator vectors for
rows and columns associated with a rectangular 0 block in $\B$, 
\[
\frac{\|\f\|_{2}^{2}\|\g\|_{2}^{2}}{n}\le\dfrac{-\min\{\boldalpha(\B)\}\sum\limits _{0\le\ell<n}\FF\left[\ell\right]\GG\left[\ell\right]e^{i\theta_{\ell}}}{d-\min\{\boldalpha(\B)\}}.
\]
\end{cor}

Note that the quantity bounded in Corollary \ref{cor:Delsarte-HoffmanDirected}
may be interpreted as the geometric average of the size of independent
set with respect to ``in'' and ``out'' nodes. Indeed, if $\f=\1_{I_{\text{row}}},\g=\1_{I_{\text{col}}}$
are the indicator functions for the in (row) and out (column) vertices
of a directed independent set, then $\|\f\|_{2}\|\g\|_{2}=\sqrt{\left|I_{\text{row}}\right|\left|I_{\text{col}}\right|}$.

If $\B$ is Hermitian and $\f=\g$, the admissibility condition (\ref{eqn:AdmissibilityConditionEigDecomp})
is always satisfied as $\theta_{\ell}=0,\ell=0,\dots,n-1$. Indeed,
in this case, 
\[
\sum_{0\le\ell<n}\FF\left[\ell\right]\GG\left[\ell\right]e^{i\theta_{\ell}}=\|\f\|_{2}^{2},
\]
so that the conclusion of Theorem \ref{thm:Delsarte-HoffmanUndirected}
holds. Hence, Corollary \ref{cor:Delsarte-HoffmanDirected} is a strict
generalization of the classical Delsarte-Hoffman inequality.

\subsubsection{Tightness of Directed Delsarte-Hoffman Bound}

When $n$ is a multiple of 4, adjacency matrices of the form 
\[
\left(\begin{array}{rrrr}
0 & 0 & 1 & 0\\
0 & 0 & 0 & 1\\
0 & 1 & 0 & 0\\
1 & 0 & 0 & 0
\end{array}\right)\otimes\1_{\frac{n}{4}\times\frac{n}{4}}
\]

show Corollary \ref{cor:Delsarte-HoffmanDirected} is tight. Indeed,
when $n=4$, this is a 1-regular directed graph with non-Hermitian
adjacency matrix $\B$, which may be decomposed as:

\[
\left(\begin{array}{rrrr}
-1/2 & -i/2 & i/2 & -1/2\\
-1/2 & -i/2 & -i/2 & -1/2\\
1/2 & 1/2 & 1/2 & -1/2\\
1/2 & -1/2 & -1/2 & -1/2
\end{array}\right)\left(\begin{array}{rrrr}
-1 & 0 & 0 & 0\\
0 & i & 0 & 0\\
0 & 0 & -i & 0\\
0 & 0 & 0 & 1
\end{array}\right)\left(\begin{array}{rrrr}
-1/2 & -i/2 & i/2 & -1/2\\
-1/2 & -i/2 & -i/2 & -1/2\\
1/2 & 1/2 & 1/2 & -1/2\\
1/2 & -1/2 & -1/2 & -1/2
\end{array}\right)^{*}.
\]

In this case, the largest independent set has size 2, corresponding
to the zero block on the upper left and lower right of the matrix.
Let $\f=\g=(1,1,0,0)^{\top}$, so that $\FF=\GG=\left(-1,0,0,-1\right)^{\top}$.
Decomposing the first and fourth eigenvalues as $\alpha_{0}=-1,\alpha_{3}=1,\theta_{0}=\theta_{3}=0$,
it is seen that the admissibility condition is satisfied, and that
$\min\{\boldalpha\}=-1$. Moreover, 
\[
\sum_{0\le\ell<4}\FF\left[\ell\right]\GG\left[\ell\right]e^{i\theta_{\ell}}=2
\]

so that the estimate of Corollary \ref{cor:Delsarte-HoffmanDirected}
is 
\[
\frac{\|\f\|_{2}^{2}\|\g\|_{2}^{2}}{4}\le1,
\]
which is tight since $\|\f\|_{2}\|\g\|_{2}=2$. A similar argument
holds for the block corresponding to indicator functions $\f=\g=(0,0,1,1)^{\top}$.
Together, this shows the maximal independent set of this directed
graph is tightly estimated by Corollary \ref{cor:Delsarte-HoffmanDirected}.

\subsection{A Delsarte-Hoffman Bound Using the Singular Value Decomposition}

Consider the singular value decomposition of $\B\in\mathbb{R}^{n\times n}$
expressed by 
\[
\B=\U\diag(\boldsigma(\B))\V^{*}\;\text{ s.t.}\;\U\U^{*}=\I=\V\V^{*},
\]
where each element of $\boldsigma(\B)$ is positive. Theorem \ref{thm:Delsarte-HoffmanSVD}
provides a Delsarte-Hoffman estimate on the size of the independent
set using the SVD, which holds for \emph{all matrices}, not just diagonalizable
ones.

\begin{thm}\label{thm:Delsarte-HoffmanSVD}Let $\B\in\mathbb{C}^{n\times n}$
be a $\delta$-regular pseudo-adjacency matrix of a directed graph.
Suppose $\B$ has a decomposition $\B=\U\diag\left(\boldsigma\right)\V^{*}$
such that $\U\U^{*}=\I=\V\V^{*}$ and $\boldsigma=(\sigma_{1},\dots,\sigma_{n})$.
Let $\sigmamin=\min_{\ell}\sigma_{\ell}$. Suppose that $\f,\g\in\{0,1\}^{n\times1}$
correspond to the indicator sets for row and column indices respectively
of a rectangular zero block, and that there exist $\FF,\GG\in\mathbb{C}^{n\times1}$
such that 
\begin{align}
\f=\U\overline{\FF},\ \g=\V\GG,\mbox{ and }\forall0\le\ell<n,\:\FF\left[\ell\right]\GG\left[\ell\right]\ge0.\label{eqn:AdmissibilityConditionSVD}
\end{align}
Then 
\[
\frac{\|\f\|_{2}^{2}\|\g\|_{2}^{2}}{n}\le\frac{-\sigmamin\underset{0\le\ell<n}{\sum}\FF\left[\ell\right]\GG\left[\ell\right]}{\delta-\sigmamin}.
\]
\end{thm}

\begin{proof} By the SVD and by $\delta$-regularity, 
\[
\B=\frac{\delta}{n}\1_{n\times n}+\left(\I-\frac{\1_{n\times n}}{n}\right)\B.
\]
Analyzing the second term for all $\f,\g\in\C^{n\times1}$, subject
to (\ref{eqn:AdmissibilityConditionSVD}), 
\begin{align*}
\left\langle \f\,\g^{*},\left(\I-\frac{\1_{n\times n}}{n}\right)\B\right\rangle = & \f^{*}\left(\I-\frac{\1_{n\times n}}{n}\right)\left(\sum_{0\le\ell<n}\sigma_{\ell}\U[:,\ell]\V^{*}[\ell,:]\right)\g\\
= & \f^{*}\left(\sum_{1\le\ell<n}\sigma_{\ell}\U[:,\ell]\V^{*}[\ell,:]\right)\g\\
= & \FF^{\top}\U^{*}\left(\sum_{1\le\ell<n}\sigma_{\ell}\U[:,\ell]\V^{*}[\ell,:]\right)\V\GG\\
= & \sum_{1\le\ell<n}\sigma_{\ell}\FF\left[\ell\right]\GG\left[\ell\right]\\
\ge & \sigmamin\sum_{1\le\ell<n}\FF\left[\ell\right]\GG\left[\ell\right]\\
= & \sigmamin\left(-\FF\left[0\right]\GG\left[0\right]+\sum_{0\le\ell<n}\FF\left[\ell\right]\GG\left[\ell\right]\right)\\
= & \sigmamin\left(-\frac{1}{n}\|\f\|_{2}^{2}\|\g\|_{2}^{2}+\sum_{0\le\ell<n}\FF\left[\ell\right]\GG\left[\ell\right]\right).
\end{align*}

Note that $\FF\left[0\right]=\frac{1}{\sqrt{n}}\|\f\|_{1}=\frac{1}{\sqrt{n}}\|\f\|_{2}^{2}$
follows from $\U^{*}\f=\U^{*}\U\overline{\FF}=\FF$ and that fact
that $\f\in\{0,1\}^{n\times1}$, $\U[:,0]=\frac{1}{\sqrt{n}}(1,\dots,1)^{\top}$;
$\GG\left[0\right]=\frac{1}{\sqrt{n}}\|\g\|_{2}^{2}$ follows similarly.
Thus, 
\begin{align*}
\left\langle \f\,\g^{*},\B\right\rangle = & \delta\left\langle \f\,\g^{*},\frac{\1_{n\times n}}{n}\right\rangle +\left\langle \f\,\g^{*},\left(\I-\frac{\1_{n\times n}}{n}\right)\B\right\rangle \\
\ge & \delta\left\langle \f\,\g^{*},\frac{\1_{n\times n}}{n}\right\rangle +\sigmamin\left(-\frac{1}{n}\|\f\|_{2}^{2}\|\g\|_{2}^{2}+\sum_{0\le\ell<n}\FF\left[\ell\right]\GG\left[\ell\right]\right)\\
= & \frac{\delta}{n}\|\f\|_{2}^{2}\|\g\|_{2}^{2}+\sigmamin\left(-\frac{1}{n}\|\f\|_{2}^{2}\|\g\|_{2}^{2}+\sum_{0\le\ell<n}\FF\left[\ell\right]\GG\left[\ell\right]\right).
\end{align*}
Every zero block $I$ is specified by a pair of indicator vectors
$\f=\1_{I_{L}}$ and $\g=\1_{I_{R}}$ such that $\f^{*}\B\g=0$. For
such an $\f$, $\g$ pair also subject to the admissibility condition
(\ref{eqn:AdmissibilityConditionSVD}) we have

\[
0\ge\delta\left(\frac{\|\f\|_{2}^{2}\|\g\|_{2}^{2}}{n}\right)-\sigmamin\left(\frac{\|\f\|_{2}^{2}\|\g\|_{2}^{2}}{n}\right)+\sigmamin\sum_{0\le\ell<n}\FF\left[\ell\right]\GG\left[\ell\right],
\]
from whence the result follows by algebraic manipulation.

\end{proof}

Theorem \ref{thm:Delsarte-HoffmanSVD} requires a decomposition which
bears resemblance to the SVD in the fact that $\B=\U\diag(\boldsymbol{\sigma})\V^{*}$,
where $\U\U^{*}=\I=\V\V^{*}$, but without the condition that $\sigma_{\ell}\ge0$
for all $\ell$. Note that if $\sigma_{\ell}\mapsto-\sigma_{\ell}$,
and $\U[:,\ell]\mapsto-\U[:,\ell]$ or $\V^{*}[\ell,:]\mapsto-\V^{*}[\ell,:]$,
this still expresses such a decomposition for $\B$. In this sense,
there are $2^{n}$ decompositions to consider in Theorem \ref{thm:Delsarte-HoffmanSVD},
corresponding to the $2^{n}$ possible sign assignments. Thus, one
can think of the decomposition in Theorem \ref{thm:Delsarte-HoffmanSVD}
as a (non-unique) signed SVD, and the condition (\ref{eqn:AdmissibilityConditionSVD})
as an admissibility condition with respect to this decomposition.

If in particular $\B$ is the adjacency matrix of a $d$-regular graph,
the following result holds.

\begin{cor}\label{cor:Delsarte-HoffmanSVD}Let $\B\in\{0,1\}^{n\times n}$
be a $d$-regular adjacency matrix of a directed graph. Suppose $\B$
has a decomposition $\B=\U\diag\left(\boldsigma\right)\V^{*}$ such
that $\U\U^{*}=\I=\V\V^{*}$ and $\boldsigma=(\sigma_{1},\dots,\sigma_{n})$.
Let $\sigmamin=\min_{\ell}\sigma_{\ell}$. Suppose that $\f\in\{0,1\}^{n\times1},\g\in\{0,1\}^{n\times1}$
correspond to the indicator sets for row and column indices respectively
of a rectangular zero block, and that there exist $\FF,\GG\in\mathbb{C}^{n\times1}$
such that 
\begin{align}
\f=\U\overline{\FF},\ \g=\V\GG,\mbox{ and }\forall0\le\ell<n,\:\FF\left[\ell\right]\GG\left[\ell\right]\ge0.\label{eqn:AdmissibilityConditionSVD}
\end{align}
Then 
\[
\frac{\|\f\|_{2}^{2}\|\g\|_{2}^{2}}{n}\le\frac{-\sigmamin\underset{0\le\ell<n}{\sum}\FF\left[\ell\right]\GG\left[\ell\right]}{d-\sigmamin}.
\]
\end{cor}


\section{Discussion and Future Research}

This article proposes generalizations of classical linear algebraic
and spectral graph theoretic results to the case in which the underlying
matrix $\B$ is non-Hermitian. This is done by constraining certain
vectors to satisfy admissibility conditions. When $\B$ is Hermitian,
these admissibility conditions hold and the classical results are
recovered. The admissibility condition take slightly different forms,
depending on which decomposition is used in place of the spectral
decomposition into an orthonormal eigenbasis.

In Theorems \ref{thm:Complex_Rayleigh_quotient}, \ref{thm:Delsarte-HoffmanDirected},
$\B$ is assumed diagonalizable as $\B=\U\diag(\boldlambda(\B))\V^{*}$
where $\boldlambda(\B)$ may be complex and $\U,\V$ need not be unitary,
merely inverses: $\U\V^{*}=\I=\V^{*}\U$. The analysis of $\f^{*}\B\g$
proceeds by assuming $\f$ admits an expansion in terms of the rows
of $\V$ and $\g$ an expansion in terms of the rows of $\U$. Of
course, when $\U=\V$ these conditions are the same, and when $\f=\g$,
this condition always holds. On the other hand, Theorem \ref{thm:Delsarte-HoffmanSVD}
takes advantage of the singular value decomposition $\B=\U\diag(\boldsigma(\B))\V^{*}$
where $\U$ and $\mathbf{V}$ are unitary but $\U\ne\V$. The analysis
of $\B$ in this situation requires a different condition on $\f,\g$,
namely that $\f$ has an admissible decomposition with respect to
the rows of $\U$, and $\g$ with respect to the rows of $\V$. We
remark that in all of these cases, the crucial property is that for
$d$-regular unweighted graphs (or $\delta$-regular weighted graphs),
the first eigenvector or singular vector (both left and right) is
the vector $\frac{1}{\sqrt{n}}(1,1,\dots,1)^{\top}\in\mathbb{R}^{n\times1}$
with corresponding eigenvalue or singular value $d$. All subsequent
analysis is downstream from this observation.

Intuitively, as $\B$ deviates from being Hermitian, the admissibility
conditions will still hold for a large class of vectors $\f,\g$.
A topic of future research is to develop a rigorous perturbation theory
of Hermitian matrices that quantifies how likely the admissibility
conditions are to hold in a probabilistic sense. That is, if $\B$
is Hermitian, then the admissibility condition holds automatically
for all $\f=\g$. As $\f$ deviates from $\g$ and $\B$ deviates
from Hermiticity, it is of interest to determine which vectors (or,
what proportion of them in a probabilistic sense) satisfy the admissibility
condition. 

\section{Acknowledgements}

We are grateful to Jim Fill (Johns Hopkins University), Yuval Filmus
(Technion), and Xiaoqin Guo (University of Wisconsin, Madison) for
insightful comments regarding the results and presentation of this
manuscript.

 \bibliographystyle{amsalpha}
\bibliography{LAA.bib}

\end{document}